\documentclass[11pt]{amsart}
\usepackage{amsmath}
\usepackage[active]{srcltx}
\usepackage{t1enc}
\usepackage[latin2]{inputenc}
\usepackage{verbatim}
\usepackage{amsmath,amsfonts,amssymb,amsthm}
\usepackage[mathcal]{eucal}
\usepackage{enumerate}
\usepackage[centertags]{amsmath}
\usepackage{graphics}
\usepackage[active]{srcltx}

\newtheorem{theorem}{Theorem}

\newtheorem{remark}{Remark}
\newtheorem{definition}{Definition}
\newtheorem{corollary}{Corollary}

\newtheorem*{A}{Theorem A}

\newtheorem*{Sah}{Theorem WS}
\newtheorem*{Bakh}{Theorem B}

\begin{document}
\author{Ushangi Goginava and Artur Sahakian}
\title[Convergence of multiple Fourier series]{On the convergence of
multiple Fourier series of functions of bounded partial generalized variation%
}
\address{U. Goginava, Department of Mathematics, Faculty of Exact and Natural
Sciences, Iv. Javakhishvili Tbilisi State University, Chavcha\-vadze str. 1, Tbilisi 0128,
Georgia}
\email{zazagoginava@gmail.com}
\address{A. Sahakian, Yerevan State University, Faculty of Mathematics and
Mechanics, Alex Manoukian str. 1, Yerevan 0025, Armenia}
\email{sart@ysu.am}
\maketitle

\begin{abstract}
The convergence of multiple Fourier series of functions of bounded partial $%
\Lambda$-variation is investigated. The sufficient and necessary conditions
on the sequence $\Lambda=\{\lambda_n\}$ are found for the convergence of
multiple Fourier series of functions of bounded partial $\Lambda$-variation.
\end{abstract}

\medskip

\footnotetext{%
}

\section{Classes of Functions of Bounded Generalized Variation}

In 1881 Jordan \cite{Jo} introduced a class of functions of bounded
variation and applied it to the theory of Fourier series. Hereafter this
notion was generalized by many authors (quadratic variation, $\Phi $%
-variation, $\Lambda $-variation ets., see \cite{Jo,Wi,W,M}). In two
dimensional case the class BV of functions of bounded variation was
introduced by Hardy \cite{Ha}.

Let $T:=[0,2\pi ]$ and
\begin{equation*}
J^{k}=\left( a^{k},b^{k}\right)\subset T,\qquad k=1,2,\ldots d.
\end{equation*}
Consider a measurable function $f\left( x\right) $ defined on $R^{d}$ and $%
2\pi$-periodic with respect to each variable. For $d=1$ we set
\begin{equation*}
f\left( J^{1}\right) :=f\left(b^{1}\right) -f\left( a^{1}\right).
\end{equation*}
If for any function of $d-1$ variables the expression $f\left( I^{1}\times
\cdots \times I^{d-1}\right) $ is already defined, then for a function of $d$
variables the \textit{mixed difference} is defined as follows:
\begin{equation*}
f\left( J^{1}\times \cdots \times J^{d}\right) :=f\left( J^{1}\times \cdots
\times J^{d-1},b^{d}\right) -f\left( J^{1}\times \cdots \times
J^{d-1},a^{d}\right) .
\end{equation*}

Let $E=\{I_{k}\}$ be a collection of nonoverlapping intervals from $T$
ordered in arbitrary way and let $\Omega $ be the set of all such
collections $E$. We denote by $\Omega _{n}$ the set of all collections of $n$
nonoverlapping intervals $I_{k}\subset T.$

For sequences of positive numbers $\Lambda^j =\{\lambda^j
_{n}\}_{n=1}^{\infty }$, $j=1,2,\ldots,d$, the $\left( \Lambda
^{1},\ldots,\Lambda ^{d}\right)$-\textit{variation of $f$ with respect to
index set }$D:=\{1,2,...,d\}$ is defined as follows:
\begin{equation*}
V_{ \Lambda ^{1},\ldots,\Lambda ^{d} }^{D}\left( f\right)
:=\sup\limits_{\{I_{i_{j}}^{j} \}_{i_j=1}^{k_j}\in \Omega }\
\sum\limits_{i_{1},...,i_{d}}\frac{\left\vert f\left( I_{i_{1}}^{1}\times
\cdots \times I_{i_{d}}^{d}\right) \right\vert }{\lambda _{i_{1}}\cdots
\lambda _{i_{d}}}.
\end{equation*}

For an index set $\alpha =\{j_{1},...,j_{p}\}\subset D$ and any $%
x=\left(x^{1},...,x^{d}\right)\in R^d $ we set ${\widetilde{\alpha}}%
:=D\setminus \alpha$ and denote by $x^{\alpha }$ the vector of $R^{p}$
consisting of components $x^{j},j\in \alpha $, i.e.
\begin{equation*}
x^{\alpha }=\left(x^{j_1},...,x^{j_p}\right)\in R^p.
\end{equation*}

By $V_{\Lambda ^{j_1},...,\Lambda ^{j_p}}^{{\alpha }}\left( f,x^{\widetilde {%
\alpha}}\right) $ and $f\left(I_{i_{j_1}}^{1}\times \cdots \times
I_{i_{j_p}}^{p},x^{\widetilde \alpha}\right)$ we denote respectively the $%
\left(\Lambda ^{j_1},...,\Lambda ^{j_p}\right)$-variation and the mixed
difference of $f$ as a function of variables $x^{j_{1}},...,x^{j_{p}}$ over
the $p$-dimensional cube $T^{p}$ with fixed values $x^{\widetilde{\alpha}}$
of other variables. The \textit{$\left(\Lambda ^{j_1},...,\Lambda
^{j_p}\right)$-variation of $f$ with respect to index set} ${\alpha }$ is
defined as follows:
\begin{equation*}
V_{\Lambda ^{j_1},...,\Lambda ^{j_p}}^{{\alpha }}\left( f\right)
=\sup\limits_{x^{{\widetilde{\alpha}}}\in T^{d-p}} V_{\Lambda
^{j_1},...,\Lambda ^{j_p}}^{{\alpha }}\left( f,x^{\widetilde{\alpha}
}\right) .
\end{equation*}

\begin{definition}
\label{def1} We say that the function $f$ has total Bounded $\left( \Lambda
^{1},...,\Lambda ^{d}\right) $-variation on $T^{d}=[0,2\pi ]^{d}$ and write $%
f\in BV_{\Lambda ^{1},...,\Lambda ^{d}}$, if
\begin{equation*}
V_{\Lambda ^{1},...,\Lambda ^{d}}(f):=\sum\limits_{\alpha \subset
D}V_{\Lambda ^{j_{1}},...,\Lambda ^{j_{p}}}^{{\alpha }}\left( f\right)
<\infty .
\end{equation*}
\end{definition}

\begin{definition}
\label{def2} We say that the function $f$ is continuous in $\left( \Lambda
^{1},...,\Lambda ^{d}\right) $-variation on $T^{d}=[0,2\pi ]^{d}$ and write $%
f\in CV_{\Lambda ^{1},...,\Lambda ^{d}}$, if%
\begin{equation*}
\lim\limits_{n\rightarrow \infty }V_{\Lambda ^{j_{1}},...,\Lambda
^{j_{k-1}},\Lambda _{n}^{j_{k}},\Lambda ^{j_{k+1}},...,\Lambda ^{j_{p}}}^{{%
\alpha }}\left( f\right) =0,\qquad k=1,2,\ldots,p
\end{equation*}%
for any $\alpha \subset D,\ \alpha :=\{j_{1},...,j_{p}\}$, where $\Lambda
_{n}^{j_{k}}:=\left\{ \lambda _{s}^{j_{k}}\right\} _{s=n}^{\infty }$.
\end{definition}

\begin{definition}
\label{def3} We say that the function $f$ has Bounded Partial $\left(
\Lambda ^{1},...,\Lambda ^{d}\right) $-variation and write $f\in
PBV_{\Lambda ^{1},...,\Lambda ^{d}}$ if
\begin{equation*}
PV_{\Lambda ^{1},...,\Lambda ^{d}}(f):=\sum\limits_{i=1}^{d}V_{\Lambda
^{i}}^{\{i\}}\left( f\right) <\infty .
\end{equation*}
\end{definition}

In the case $\Lambda ^{1}=\cdots =\Lambda ^{d}=\Lambda $ we denote%
\begin{equation*}
BV_{\Lambda }:=BV_{\Lambda^{1},...,\Lambda ^{d}},\quad CV_{\Lambda
}:=CV_{\Lambda ^{1},...,\Lambda^{d}},\quad PBV_{\Lambda }:=PBV_{\Lambda
^{1},...,\Lambda ^{d}}
\end{equation*}%
and%
\begin{equation*}
CV_{\Lambda }:=V_{\Lambda ^{1},...,\Lambda ^{d}}(f)CV_{\Lambda },\qquad
PV_{\Lambda }(f):=PV_{\Lambda ^{1},...,\Lambda ^{d}}(f).
\end{equation*}

If $\lambda _{n}\equiv 1$ (or if $0<c<\lambda _{n}<C<\infty ,\ n=1,2,\ldots $%
) the classes $BV_{\Lambda }$ and $PBV_{\Lambda }$ coincide with the Hardy
class $BV$ and $PBV$ respectively. Hence it is reasonable to assume that $%
\lambda _{n}\rightarrow \infty $ and since the intervals in $E=\{I_{i}\}$
are ordered arbitrarily, we suppose, without loss of generality, that the
sequence $\{\lambda _{n}\}$ is increasing. Thus,
\begin{equation}
1<\lambda _{1}\leq \lambda _{2}\leq \ldots ,\qquad \lim_{n\rightarrow \infty
}\lambda _{n}=\infty .  \label{Lambda}
\end{equation}

When $\lambda _{n}=n$ for all $n=1,2\ldots $ we say \textit{Harmonic
Variation} instead of $\Lambda $-variation and write $H$ instead of $\Lambda
$ ($BV_{H}$, $PBV_{H}$, $CV_H$, ets).

\begin{remark}
The notion of $\Lambda $-variation was introduced by Waterman \cite{W} in
one dimensional case, by Sahakian \cite{Saha} in two dimensional case and by
Sablin \cite{Sab} in the case of higher dimensions. The notion of bounded
partial variation (class $PBV$) was introduced by Goginava in \cite%
{GoEJA,GoJAT}. These classes of functions of generalized bounded variation
play an important role in the theory Fourier series.

Observe, that the number of variations in Definition \ref{def1} of total
variation is $2^d-1$, while the number of variations in Definition \ref{def2}
of partial variation is only $d$.
\end{remark}

The statements of the following theorem are known.

\begin{A}
\label{A} 1) \textrm{(Dragoshanski \cite{Dr})} If $d=2$, then $BV_{H}=CV_{H}$%
.

2) \textrm{(Bakhvalov \cite{Bakh1})} $CV_H=\bigcup_\Gamma BV_\Gamma$ for any
$d$, where the union is taken over all sequences $\Gamma =\{\gamma
_{n}\}_{n=1}^{\infty }$ with $\gamma _{n}=o(n)$ as $n\rightarrow \infty $.

3) \textrm{(Goginava, Sahakian \cite{GogSah})} If $d=2$, then $PBV_{\Lambda
}\subset BV_{H}$, provided that
\begin{equation}  \label{lambda d=2}
\frac{\lambda _{n}}{n}\downarrow 0\quad \mathrm{{and}\quad
\sum\limits_{n=1}^{\infty }\frac{\lambda _{n}}{n^{2}}<\infty ,}
\end{equation}
\end{A}

Using the third statement of Theorem A, we have proved in \cite{GogSah} the
convergence of double Fourier series of functions of any class $PBV_{\Lambda
}$ with (\ref{lambda d=2}). To obtain similar result for higher dimensions
we need stronger result, since the inclusion $PBV_{\Lambda }\subset BV_{H}$
is not enough in this case (see next section for details).

\begin{theorem}
\label{t1} Let $\Lambda =\{\lambda _{n}\}_{n=1}^{\infty }$ and $d\geq 2$. If
\begin{equation}  \label{lambda}
\frac{\lambda _{n}}{n}\downarrow 0\quad \mathrm{{and}\quad
\sum\limits_{n=1}^{\infty }\frac{\lambda _{n}\log ^{d-2}n}{n^{2}}<\infty ,}
\end{equation}%
then there exists a sequence $\Gamma =\{\gamma _{n}\}_{n=1}^{\infty }$ with
\begin{equation}  \label{gamma0}
\gamma _{n}=o(n)\quad \mathrm{as}\quad n\rightarrow \infty,
\end{equation}
such that $PBV_{\Lambda }\subset BV_{\Gamma }$.
\end{theorem}

\begin{proof}
Choosing the sequence $\{A_n\}_{n=1}^\infty$ such that
\begin{equation}  \label{an}
A_n\uparrow\infty,\qquad \frac{\lambda_nA_n}n\downarrow0,\qquad
\sum\limits_{n=1}^{\infty }\frac{\lambda _{n}\log ^{d-2} n A_n^d}{n^{2}}%
<\infty,
\end{equation}
we set
\begin{equation}  \label{gamma}
\gamma_n=\frac n{A_n},\qquad n=1,2,\ldots
\end{equation}

We prove that there is a constant $C>0$ such that
\begin{equation}  \label{sup}
\sum_{i_{1},\ldots ,i_{p}}\frac{\left\vert f\left( I_{i_{1}}^{1}\times
\cdots \times I_{i_{p}}^{p},x^{\widetilde \alpha}\right) \right\vert }{%
\gamma _{i_{1}}\cdots \gamma _{i_{p}}}<C\cdot PV_{\Lambda }(f),
\end{equation}
for any $f\in PBV_{\Lambda }$ and $\alpha :=\{i_{1},...,i_{p}\}\subset D$, ${%
\{I_{i_{j}}^{j}\}_{i_j=1}^{k_j}\in \Omega}$.

To prove (\ref{sup}) observe, that
\begin{eqnarray}  \label{sum}
&&\sum_{i_{1},\ldots ,\,i_{p}}\frac{\left\vert f\left( I_{i_{1}}^{1}\times
\cdots \times I_{i_{p}}^{p},x^{\widetilde \alpha}\right) \right\vert }{%
\gamma _{i_{1}}\cdots \gamma _{i_{p}}} \\
&&= \sum_{\sigma }\sum_{i_{\sigma (1)}\leq \cdots \leq i_{\sigma (p)}}\frac{%
\left\vert f\left( I_{i_{1}}^{1}\times \cdots \times
I_{i_{p}}^{p},x^{\widetilde \alpha}\right) \right\vert }{\gamma
_{i_{1}}\cdots \gamma _{i_{p}}}<\infty ,  \notag
\end{eqnarray}
where the sum is taken over all rearrangements $\sigma =\{\sigma
(k)\}_{k=1}^{p}$ of the set $\{1,2,\ldots ,p\}$.

Denoting $M=PV_{\Lambda }(f)$ and using (\ref{gamma}), (\ref{an}) and (\ref%
{lambda}) we obtain:
\begin{eqnarray*}
&&\sum\limits_{i_{1}\leq i_{2}\leq \cdots \leq i_{p}}\frac{\left\vert
f\left( I_{i_{1}}^{1}\times \cdots \times I_{i_{p}}^{p},x^{\widetilde{\alpha
}}\right) \right\vert }{\gamma _{i_{1}}\cdots \gamma _{i_{p}}} \\
&=&\sum\limits_{i_{1}\leq i_{2}\leq \cdots \leq i_{p-1}}\frac{%
A_{i_{1}}\cdots A_{i_{p-1}}}{{i_{1}}\cdots {i_{p-1}}}\sum\limits_{i_{p}\geq
i_{p-1}}\frac{\left\vert f\left( I_{i_{1}}^{1}\times \cdots \times
I_{i_{p}}^{p},x^{\widetilde{\alpha }}\right) \right\vert }{\lambda _{i_{p}}}%
\cdot \frac{\lambda _{i_{p}}A_{i_{p}}}{i_{p}} \\
&\leq &M\sum\limits_{i_{1}\leq i_{2}\leq \cdots \leq i_{p-1}}\frac{%
A_{i_{p-1}}^{p}\lambda _{i_{p-1}}}{i_{p-1}^{2}}\cdot \frac{1}{{i_{1}}\cdots
i_{p-2}} \\
&=&M\sum\limits_{i_{i_{p}-1}=1}^{\infty }\frac{A_{i_{p-1}}^{p}\lambda
_{i_{p-1}}}{i_{p-1}^{2}}\sum\limits_{i_{p-2}=1}^{i_{p-1}}\frac{1}{i_{p-2}}%
\sum\limits_{i_{p-3}=1}^{i_{p-2}}\frac{1}{i_{p-3}}\cdots
\sum\limits_{i_{1}=1}^{i_{2}}\frac{1}{{i_{1}}} \\
&\leq &M\sum\limits_{i_{p-1}=1}^{\infty }\frac{A_{i_{p-1}}^{p}\lambda
_{i_{p-1}}}{i_{p-1}^{2}}\left( \sum\limits_{i_{=}1}^{i_{p-1}}\frac{1}{i}%
\right) ^{p-2}\leq C\cdot M\sum\limits_{n=1}^{\infty }\frac{A_{n}^{p}\lambda
_{n}\log ^{d-2}n}{n^{2}}<\infty .
\end{eqnarray*}%
Similarly we can prove that all other summands in the right hind side of (%
\ref{sum}) are finite. Theorem \ref{t1} is proved.
\end{proof}

In view of Theorem A, Theorem \ref{t1} implies

\begin{corollary}
\label{c1} If the sequence $\Lambda =\{\lambda _{n}\}_{n=1}^{\infty }$
satisfies (\ref{lambda}), then $PBV_{\Lambda }\subset CV_{H}$.
\end{corollary}

\begin{definition}
The partial modulus of variation $v_{i}\left( n,f\right) $, $i=1,\ldots,d$
of a function $f$ are defined by
\begin{equation*}
v_{i}\left( n,f\right) :=\sup\limits_{x^{\beta }}\sup\limits_{\{I_{j}\}\in
\Omega _{n}}\sum\limits_{j=1}^{n}\left\vert f\left( I_{j},x^{\beta }\right)
\right\vert ,\quad \beta=D\setminus\{i\},\quad n=1,2,\ldots .
\end{equation*}
\end{definition}

For functions of one variable the concept of modulus of variation was
introduced by Chanturia \cite{Ch}.

\begin{theorem}
\label{t2} Let $f$ be defined on $T^{d}$ and%
\begin{equation}
\sum\limits_{j=1}^{\infty }\frac{\sqrt[d]{v_{i}\left( 2^{j},f\right) }}{%
2^{j/d}}<\infty ,\qquad i=1,...,d.  \label{var}
\end{equation}%
Then there exists a sequence $\Delta =\{\delta _{n}\}_{n=1}^{\infty }$ with
\begin{equation*}
\delta _{n}=o\left( n\right)\quad \mathrm{as}\quad n\to \infty,
\end{equation*}
such that $f\in BV_{\Delta }.$
\end{theorem}

\begin{proof}
We use induction on dimension $d$. We have proved in \cite{GogSah}, that in
the case $d=2$ the condition (\ref{var}) implies $f\in BV_{H}$, which
combined with Theorem A proves Theorem \ref{t2} for $d=2$.

Supposing Theorem \ref{t2} is true if the dimension is less than $d$, we
prove it for the dimension $d>2$.

According to induction hypothesis it is enough to prove that there exists a
sequence 
$\delta _{n}=o(n)$ such that
\begin{equation*}
\sup_{\{I_{i_{j}}^{j}\}_{i_j=1}^{k_j}\in\Omega}\ \sum_{i_{1},\ldots ,i_{d}}%
\frac{\left\vert f\left( I_{i_{1}}^{1}\times \cdots \times
I_{i_{d}}^{d}\right) \right\vert }{\delta _{i_{1}}\cdots \delta _{i_{d}}}%
<\infty .
\end{equation*}
Let the sequence $\{B_{2^{j}}\}_{j=1}^{\infty }$ be chosen so that
\begin{equation*}
B_{2^{j}}\uparrow \infty ,\qquad\sum\limits_{j=1}^{\infty }\frac{B_{2^{j}}%
\sqrt[d]{v_{i}\left( 2^{j},f\right) }}{2^{j/d}}<\infty ,\qquad i=1,...,d.
\end{equation*}%
Defining
\begin{equation*}
B_{n}=B_{2^{N}},\quad \mathrm{for}\quad 2^{N}\leq n<2^{N+1},\qquad
N=0,1,.....,
\end{equation*}
we set
\begin{equation}
\delta _{n}=\frac{n}{B_{n}},\qquad n=1,2\ldots  \label{delta}
\end{equation}
Then we can write
\begin{eqnarray}
&&\sum_{i_{1},\ldots ,i_{d}}\frac{\left\vert f\left( I_{i_{1}}^{1}\times
\cdots \times I_{i_{d}}^{d}\right) \right\vert }{\delta _{i_{1}}\cdots
\delta _{i_{d}}}  \label{d-var} \\
&=&\sum_{i_{1},\ldots ,i_{d}}B_{i_{1}}\cdots B_{i_{d}}\frac{\left\vert
f\left( I_{i_{1}}^{1}\times \cdots \times I_{i_{d}}^{d}\right) \right\vert }{%
i_{1}\cdots i_{d}}  \notag \\
&=&\sum\limits_{r_{1}=0}^{\infty }\cdots \sum\limits_{r_{d}=0}^{\infty
}\sum\limits_{i_{1}=2^{r_{1}}}^{2^{r_{1}+1}-1}\cdots
\sum\limits_{i_{d}=2^{r_{d}}}^{2^{r_{d}+1}-1}B_{i_{1}}\cdots B_{i_{d}}\frac{%
\left\vert f\left( I_{i_{1}}^{1}\times \cdots \times I_{i_{d}}^{d}\right)
\right\vert }{i_{1}\cdots i_{d}}  \notag \\
&\leq &\sum\limits_{r_{1}=0}^{\infty }\cdots \sum\limits_{r_{d}=0}^{\infty }%
\frac{B_{2^{r_{1}}}}{2^{r_{1}}}\cdots \frac{B_{2^{r_{d}}}}{2^{r_{d}}}%
\sum\limits_{i_{1}=2^{r_{1}}}^{2^{r_{1}+1}-1}\cdots
\sum\limits_{i_{d}=2^{r_{d}}}^{2^{r_{d}+1}-1}\left\vert f\left(
I_{i_{1}}^{1}\times \cdots \times I_{i_{d}}^{d}\right) \right\vert .  \notag
\end{eqnarray}

It is easy to show that%
\begin{eqnarray*}
&&\sum\limits_{i_{1}=2^{r_{1}}}^{2^{r_{1}+1}-1}\cdots
\sum\limits_{i_{d}=2^{r_{d}}}^{2^{r_{d}+1}-1}\left\vert f\left(
I_{i_{1}}^{1}\times \cdots \times I_{i_{d}}^{d}\right) \right\vert \\
&\leq &c\left( d\right) \prod\limits_{i\in \beta
}2^{r_{i}}\sup\limits_{x^{\beta }}\sup\limits_{\{I_{i_{k}}^{k}\}\in \Omega
_{2^{r_{k}}}}\sum\limits_{i_{k}=2^{r_{k}}}^{2^{r_{k}+1}-1}\left\vert f\left(
I_{i_{k}}^{k},x^{\beta }\right) \right\vert,
\end{eqnarray*}
where $\beta :=D\setminus\{k\}$, $k=1,...,d$.

Consequently,
\begin{eqnarray}
&&\sum\limits_{i_{1}=2^{r_{1}}}^{2^{r_{1}+1}-1}\cdots
\sum\limits_{i_{d}=2^{r_{d}}}^{2^{r_{d}+1}-1}\left\vert f\left(
I_{i_{1}}^{1}\times \cdots \times I_{i_{d}}^{d}\right) \right\vert
\label{p-var} \\
&=&\left[ \left( \sum\limits_{i_{1}=2^{r_{1}}}^{2^{r_{1}+1}-1}\cdots
\sum\limits_{i_{d}=2^{r_{d}}}^{2^{r_{d}+1}-1}\left\vert f\left(
I_{i_{1}}^{1}\times \cdots \times I_{i_{d}}^{d}\right) \right\vert \right)
^{1/d}\right] ^{d}  \notag \\
&\leq &c\left( d\right) \prod\limits_{k=1}^{d}2^{r_{k}\left( 1-1/d\right)
}\left( \sup\limits_{x^{\beta }}\sup\limits_{\{I_{i_{k}}^{k}\}\in \Omega
_{2^{r_{k}}}}\sum\limits_{i_{k}=2^{r_{k}}}^{2^{r_{k}+1}-1}\left\vert f\left(
I_{i_{k}}^{k},x^{\beta }\right) \right\vert \right) ^{1/d}  \notag \\
&=&c\left( d\right) \prod\limits_{k=1}^{d}2^{r_{k}\left( 1-1/d\right) }\sqrt[%
d]{v_{k}\left( 2^{r_{k}},f\right) }.  \notag
\end{eqnarray}

Combining (\ref{d-var}) and (\ref{p-var}) we obtain%
\begin{eqnarray*}
&&\sum_{i_{1},\ldots ,i_{d}}\frac{\left\vert f\left( I_{i_{1}}^{1}\times
\cdots \times I_{i_{d}}^{d}\right) \right\vert }{\delta _{i_{1}}\cdots
\delta _{i_{d}}} \\
&\leq &c\left( d\right) \sum\limits_{r_{1}=0}^{\infty }\cdots
\sum\limits_{r_{d}=0}^{\infty }\frac{B_{2^{r_{1}}}v_{1}\left(
2^{r_{1}},f\right) }{2^{r_{1}/d}}\cdots \frac{B_{2^{r_{d}}}v_{d}\left(
2^{r_{d}},f\right) }{2^{r_{d}}/d}<\infty .
\end{eqnarray*}

Theorem 2 is proved.
\end{proof}

\section{\protect\medskip Convergence of multiple Fourier series}

The Fourier series of function $f\in L^{1}\left( T^{d}\right)$ with respect
to the trigonometric system is the series
\begin{equation*}
S\left[ f\right] :=\sum_{n_{1},...,n_{d}=-\infty }^{+\infty }\widehat{f}%
\left( n_{1},....,n_{d}\right) e^{i\left( n_{1}x+\cdots +n_{d}x_{d}\right) },
\end{equation*}%
where
\begin{equation*}
\widehat{f}\left( n_{1},....,n_{d}\right) =\frac{1}{\left( 2\pi \right) ^{d}}%
\int_{T^{d}}f(x^{1},...,x^{d})e^{-i\left( n_{1}x^{1}+\cdots
+n_{d}x^{d}\right) }dx^{1}\cdots dx^{d}
\end{equation*}
are the Fourier coefficients of $f$. The rectangular partial sums are
defined as follows:
\begin{eqnarray*}
&&S_{N_{1},...,N_{d}}\left( f;x^{1},...,x^{d}\right) \\
&&:=\sum_{n_{1}=-N_{1}}^{N_{1}}\cdots \sum_{n_{d}=-N_{d}}^{N_{d}}\widehat{f}%
\left( n_{1},....,n_{d}\right) e^{i\left( n_{1}x^{1}+\cdots
+n_{d}x^{d}\right) } \\
&&=\frac 1{\pi^d}\int\limits_{T^{d}}f\left( x_{1},\cdots,x_{d}\right)
\prod\limits_{s=1}^{d}D_{N_s}\left( x_{s}\right) dx_{1}\cdots dx_{d},
\end{eqnarray*}
where $D_N(t)=\frac{sin\left(N+\frac12 \right)t}{2sin\frac t2}$ is the
Dirichlet kernel.

In this paper we consider convergence of \textbf{only rectangular partial
sums} (convergence in the sense of Pringsheim) of $d$-dimensional Fourier
series.

We denote by $C(T^{d})$ the space of continuous and $2\pi $-periodic with
respect to each variable functions with the norm
\begin{equation*}
\Vert f\Vert _{C}:=\sup_{\left( x^{1},\ldots,\,x^{d}\right) \in
T^{d}}|f(x^{1},\ldots,x^{d})|.
\end{equation*}

We say that the point $x:=\left( x^{1},\ldots,x^{d}\right) $ is \textit{a
regular point} of function $f$ if the following limits exist
\begin{equation*}
f\left( x^{1}\pm 0,...,x^{d}\pm 0\right)
:=\lim\limits_{t^{1},\ldots,\,t^{d}\downarrow 0}f\left( x^{1}\pm
t^{1},\ldots,x^{d}\pm t^{d}\right) .
\end{equation*}%
For the regular point $x:=\left( x^{1},\ldots,x^{d}\right) $ we denote%
\begin{equation}
f^{\ast }\left( x^{1},\ldots,x^{d}\right) :=\frac{1}{2^{d}}\sum f\left(
x^{1}\pm 0,\ldots,x^{d}\pm 0\right) .  \label{limit}
\end{equation}

\begin{definition}
We say that the class of functions $V \subset L^{1}(T^{d})$ is a class of
convergence on $T^{d}$, if for any function $f\in V $

1) the Fourier series of $f$ converges to $f^{\ast }({x})$ at any regular
point ${x}\in T^{d}$,

2) the convergence is uniform on any compact $K\subset T^{d}$, if $f$ is
continuous on the neighborhood of $K$.
\end{definition}

The well known Dirichlet-Jordan theorem (see \cite{Zy}) states that the
Fourier series of a function $f(x), \ x\in T$ of bounded variation converges
at every point $x$ to the value $\left[ f\left( x+0\right)
+f\left(x-0\right) \right] /2$. If $f$ is in addition continuous on $T$, the
Fourier series converges uniformly on $T$.

Hardy \cite{Ha} generalized the Dirichlet-Jordan theorem to the double
Fourier series and proved that $BV$ is a class of convergence on $T^{2}$.

The following theorem was proved by Waterman (for $d=1$) and Sahakian (for $%
d=2$).

\begin{Sah}[Waterman \protect\cite{W}, Sahakian \protect\cite{Saha}]
If $d=1$ or $d=2$, then the class $BV_{H}$ is a class of convergence on $%
T^{d}$.
\end{Sah}

In \cite{Bakh1} Bakhvalov showed that the class $BV_{H}$ is not a class of
convergence on $T^{d}$, if $d>2$. On the other hand, he proved the following

\begin{Bakh}[Bakhvalov \protect\cite{Bakh1}]
\label{B} The class $CV_{H}$ is a class of convergence on $T^{d}$ for any $%
d=1,2,\ldots$
\end{Bakh}

Convergence of spherical and other partial sums of double Fourier series of
functions of bounded $\Lambda $-variation was investigated in deatails by
Dyachenko \cite{D1,D2}.

The main result of this paper is the following theorem, that  we have proved
in \cite{GogSah} for $d=2$.

\begin{theorem}
\label{t3} Let $\Lambda =\{\lambda _{n}\}_{n=1}^{\infty }$ and $d\geq 2$.

a) If
\begin{equation}  \label{lambda1}
\sum\limits_{n=1}^{\infty }\frac{\lambda _{n}\log ^{d-2}n}{n^{2}}<\infty ,
\end{equation}%
then $PBV_{\Lambda }$ is a class of convergence on $T^{d}$.

\medskip b) If
\begin{equation}  \label{lambda2}
\frac{\lambda _{n}}{n}=O\left( \frac{\lambda _{\lbrack n^{\delta }]}}{%
[n^{\delta }]}\right)
\end{equation}%
for some $\delta >1$, and
\begin{equation}  \label{lambda3}
\sum\limits_{n=1}^{\infty }\frac{\lambda _{n}\log ^{d-2}n}{n^{2}}=\infty ,
\end{equation}%
then there exists a continuous function $f\in PBV_{\Lambda }$, the Fourier
series of which diverges at $\left( 0,\ldots,0\right) .$
\end{theorem}

\begin{proof}[Proof of Theorem 2]
Part a) immediately follows from Corollary \ref{c1} and Theorem B.

To prove part b) we denote
\begin{equation*}
A_{i_{1},\ldots,i_{d}}:=\left[ \frac{\pi i_{1}}{N+1/2},\frac{\pi \left(
i_{1}+1\right) }{N+1/2}\right) \times \cdots \times \left[ \frac{\pi i_{d}}{%
N+1/2},\frac{\pi \left( i_{d}+1\right) }{N+1/2}\right) ,
\end{equation*}
\medskip
\begin{equation*}
W :=\left\{ (i_{1},\ldots ,i_{d}):i_{d}<i_{s}<i_{d}+m_{i_{d}},\ 1\leq s<d,\
1\leq i_{d}\leq N_{\delta }\right\} ,\quad
\end{equation*}
\medskip
\begin{equation*}
N_{\delta } =\left[ \left( \frac{N}{2}\right) ^{\frac{1}{\delta }}\right],
\qquad t_{j}:=\left( \sum\limits_{i=1}^{m_{j}}\frac{1}{\lambda _{i}}%
\right)^{-1},\qquad m_{j}:=\left[ j^{\delta }\right],
\end{equation*}
where $[x]$ is the integer part of $x$.

It is not hard to see, that for any sequence $\Lambda =\{\lambda _{n}\}$
satisfying (\ref{Lambda}) the class $C(T^{d})\cap PBV_{\Lambda }$ is a
Banach space with the norm
\begin{equation*}
\Vert f\Vert _{PBV_{\Lambda }}:=\Vert f\Vert _{C}+PV_{\Lambda }(f).
\end{equation*}

Consider the following function%
\begin{equation*}
f_{N}\left( x_{1},\ldots,x_{d}\right) :=\sum\limits_{\left(
i_{1},\ldots,\,i_{d}\right) \in W}t_{i_{d}}1_{A_{i_{1},\ldots,i_{d}}}\left(
x_{1},\ldots,x_{d}\right) \prod\limits_{s=1}^{d}\sin \left( N+1/2\right)
x_{s},
\end{equation*}%
where $1_{A}\left( x_{1},\ldots,x_{d}\right) $ is the characteristic
function of the set $A\subset T^{d}$.

Let $\left( i_{1},\ldots ,i_{k-1},i_{k+1},\ldots ,i_{d}\right) $ be fixed $%
(k=1,\ldots ,d-1)$. Then it is easy to show that
\begin{equation*}
V_{\Lambda }^{{k}}\left( f_{N}\right) \leq C\cdot t_{i_{d}}\left(
\sum\limits_{i_{k}=i_{d}+1}^{i_{d}+m_{i_{d}}}\frac{1}{\lambda _{i_{k}-i_{d}}}%
\right) \leq C\cdot t_{i_{d}}\left( \sum\limits_{i_{k}=1}^{m_{i_{d}}}\frac{1%
}{\lambda _{i_{k}}}\right) \leq C<\infty .
\end{equation*}%
If $\left( i_{1},\ldots ,i_{d-1}\right) $ is fixed, the condition $\left(
i_{1},\ldots ,i_{d}\right) \in W$ implies
\begin{equation*}
\max \left\{ i_{d}\left( i_{s}\right) :1\leq s\leq d-1\right\} <i_{d}<\min
\left\{ i_{s}:1\leq s\leq d-1\right\} ,
\end{equation*}%
where%
\begin{equation*}
i_{d}\left( i_{s}\right) :=\min \left\{ i_{d}:i_{d}+m_{i_{d}}>i_{s}\right\} .
\end{equation*}%
Consequently, by the definition of the function $f_{N}$ we obtain that for
any $s=1,...,d-1$
\begin{eqnarray*}
V_{\Lambda }^{{d}}\left( f_{N}\right)  &\leq
&C\sum\limits_{i_{d}=i_{d}\left( i_{s}\right) +1}^{i_{s}}\frac{t_{i_{d}}}{%
\lambda _{i_{d}-i_{d}\left( i_{s}\right) }} \\
&\leq &C\cdot t_{i_{d}\left( i_{s}\right) }\sum\limits_{i_{d}=i_{d}\left(
i_{s}\right) +1}^{i_{s}}\frac{1}{\lambda _{i_{d}-i_{d}\left( i_{s}\right) }}
\\
&=&C\cdot t_{i_{d}\left( i_{s}\right)
}\sum\limits_{i_{d}=1}^{i_{s}-i_{d}\left( i_{s}\right) }\frac{1}{\lambda
_{i_{d}}}\leq C\cdot t_{i_{d}\left( i_{s}\right)
}\sum\limits_{i_{d}=1}^{m_{i_{d}\left( i_{s}\right) }}\frac{1}{\lambda
_{i_{d}}}=C<\infty .
\end{eqnarray*}%
Hence $f_{N}\in PBV_{\Lambda }$ and
\begin{equation}
\Vert f_{N}\Vert _{PV_{\Lambda }}\leq C,\quad N=1,2,\ldots .  \label{ineq1}
\end{equation}%
Observe, that by (\ref{lambda2}) we have%
\begin{equation*}
\frac{1}{t_{j}}=\sum_{i=1}^{m_{j}}\frac{1}{\lambda _{i}}=\sum_{i=1}^{m_{j}}%
\frac{1}{i}\cdot \frac{i}{\lambda _{i}}\leq C\frac{m_{j}}{\lambda _{m_{j}}}%
\log m_{j}\leq C\frac{j\log j}{\lambda _{j}}.
\end{equation*}%
Hence%
\begin{equation*}
t_{j}\log j\geq c\frac{\lambda _{j}}{j}.
\end{equation*}%
Consequently,
\begin{eqnarray}
&&\pi ^{d}S_{N,\cdots ,N}\left( f_{N};0,\cdots ,0\right)   \label{ineq3} \\
&=&\int\limits_{T^{d}}f_{N}\left( x^{1},\cdots ,x^{d}\right)
\prod\limits_{s=1}^{d}D_{N}\left( x^{s}\right) dx^{1}\cdots dx^{d}  \notag
\label{ineq2} \\
&=&\sum\limits_{\left( i_{1},\cdots ,i_{d}\right) \in
W}t_{i_{d}}\int\limits_{A_{i_{1},\cdots ,\,i_{d}}}\prod\limits_{s=1}^{d}%
\frac{\sin ^{2}\left( N+1/2\right) x^{s}}{2\sin \left( x^{s}/2\right) }%
dx^{1}\cdots dx^{d}  \notag \\
&\geq &c\sum\limits_{\left( i_{1},\cdots ,\,i_{d}\right) \in W}t_{i_{d}}%
\frac{1}{i_{1}\cdots i_{d}}  \notag \\
&\geq &c\sum\limits_{i_{d}=1}^{N_{\delta }}\frac{t_{i_{d}}}{i_{d}}%
\sum\limits_{i_{1}=i_{d}}^{i_{d}+m_{i_{d}}}\cdots
\sum\limits_{i_{d-1}=i_{d}}^{i_{d}+m_{i_{d}}}\frac{1}{i_{1}\cdots i_{d-1}}
\notag \\
&\geq &c\sum\limits_{i_{d}=1}^{N_{\delta }}\frac{t_{i_{d}}}{i_{d}}\log
^{d-1}\left( \frac{i_{d}+m_{i_{d}}}{i_{d}}\right)   \notag \\
&\geq &c(\delta -1)^{d-1}\sum\limits_{i_{d}=1}^{N_{\delta }}\frac{%
t_{i_{d}}\log i_{d}}{i_{d}}\log ^{d-2}i_{d}  \notag \\
&\geq &c(\delta -1)^{d-1}\sum\limits_{n=1}^{N_{\delta }}\frac{\lambda
_{n}\log ^{d-2}n}{n^{2}}\rightarrow \infty ,  \notag
\end{eqnarray}%
as $N\rightarrow \infty $, according to (\ref{lambda3}).

By Banach-Steinhaus Theorem, (\ref{ineq1}) and (\ref{ineq3}) imply the
existence of a continuous function $f\in PBV_{\Lambda }$ such that
\begin{equation*}
\sup_{N}|S_{N,\cdots,N}[f,(0,\cdots,0)]|=\infty .
\end{equation*}
\end{proof}

\begin{corollary}
\label{c3} a) If $\Lambda =\left\{ \lambda _{n}\right\} _{n=1}^{\infty }$
with
\begin{equation*}
\lambda _{n}=\frac{n}{\log ^{d-1+\varepsilon }n},\qquad n=1,2,\ldots
\end{equation*}%
for some $\varepsilon >0$, then the class $PBV_{\Lambda }$ is a class of
convergence on $T^{d}$.

\medskip b) If $\Lambda =\left\{ \lambda _{n}\right\} _{n=1}^{\infty }$ with
\begin{equation*}
\lambda _{n}=\frac{n}{\log ^{d-1}n},\qquad n=1,2,\ldots ,
\end{equation*}%
then the class $PBV_{\Lambda }$ is not a class of convergence on $T^{d}$.
\end{corollary}

The second part of Theorem \ref{t2} and Corollary \ref{c1} imply

\begin{corollary}
\label{c2} If the sequence $\Lambda =\{\lambda _{n}\}_{n=1}^{\infty }$
satisfies (\ref{lambda2}) and (\ref{lambda3}), then $PBV_{\Lambda
}\not\subset CV_{H}$.
\end{corollary}

Theorem \ref{t2} and Theorems A and B imply

\begin{theorem}
The set of functions
\begin{equation*}
\left\{ f:\sum\limits_{j=0}^{\infty }\frac{\sqrt[d]{v_{i}\left(
2^{j},f\right) }}{2^{j/d}}<\infty ,\ i=1,...,d\right\}
\end{equation*}%
is a class of convergence on $T^{d}$.
\end{theorem}

\begin{corollary}
The set of functions
\begin{equation*}
\left\{ f:v_{i}\left( n,f\right) =O\left( n^{\alpha }\right),\
i=1,...,d\right\}
\end{equation*}%
is a class of convergence on $T^{d}$ for any $\alpha\in(0,1)$.
\end{corollary}

\end{document}